\documentclass{article}
\usepackage{amsmath,amsthm,amssymb,mathrsfs,amstext, titlesec,enumitem}
\usepackage[numbers]{natbib}
\makeatletter

\titleformat{\section}
{\normalfont\Large\bfseries}{\thesection}{1em}{}
\titleformat*{\subsection}{\Large\bfseries}
\titleformat*{\subsubsection}{\large\bfseries}
\titleformat*{\paragraph}{\large\bfseries}
\titleformat*{\subparagraph}{\large\bfseries}

\renewcommand{\@seccntformat}[1]{\csname the#1\endcsname.}

\renewenvironment{abstract}{%
    \if@twocolumn
      \section*{\abstractname}%
    \else 
      \begin{center}%
        {\bfseries \Large\abstractname\vspace{\z@}}
      \end{center}%
      \quotation
    \fi}
    {\if@twocolumn\else\endquotation\fi}

\makeatother
\theoremstyle{plain}
\newtheorem{thm}{Theorem}[section]
\newtheorem{lem}[thm]{Lemma}
\theoremstyle{definition}
\newtheorem{defn}[thm]{Definition}
\newtheorem{rem}[thm]{Remark}
\newtheorem{fact}[thm]{Fact}
\newtheorem{cor}[thm]{Corollary}
\newtheorem{example}[thm]{Example}
\providecommand{\keywords}[1]{{\bf{Keywords:}} #1}
\providecommand{\subjectclass}[2]{\textbf{Mathematics subject classification 2020:} #1}

\title{{\bf Hales-Jewett type configurations in small sets}}

\author {          
Aninda Chakraborty \footnote{Department of Mathematics, Government General Degree College at Chapra, Chapra, Nadia, West Bengal, India.
          }
      \\ {\tt anindachakraborty2@gmail.com} 

           \and
Sayan Goswami \footnote{Corresponding author}
\footnote{Department of Mathematics, 
          University of Kalyani, 
          Kalyani-741235,
          Nadia, West Bengal, India
          } 
      \\ {\tt sayan92m@gmail.com}

}
\date{\vspace{-5ex}}

\begin{document}
\maketitle

\begin{abstract}
\noindent In a recent work, N. Hindman, D. Strauss and L. Zamboni have shown
that the Hales-Jewett theorem can be combined with a sufficiently
well behaved homomorphisms. Their work was completely algebraic in
nature, where they used the algebra of Stone-\v{C}ech compactification
of discrete semigroups. They proved the existence of those configurations
in piecewise syndetic sets, which are Ramsey theoretic rich sets.
In our work we will show those forms are still present in very small
but Ramsey theoretic sets, (like $J$-sets, $C$-sets) and our proof
is purely elementary in nature.\\

\noindent \subjectclass{05D10} \\

\noindent \keywords{Hales-Jewett theorem}
\end{abstract}

\section{Introduction}

Let $\omega=\mathbb{N}\cup\left\{ 0\right\} $, where $\mathbb{N}$
is the set of positive integers. Then $\omega$ is the first infinite
ordinal. For any set $X$, let $\mathcal{P}_{f}\left(X\right)$ be
the set of all nonempty finite subsets of  $X$.

Given a nonempty set $\mathbb{A}$ (or alphabet) we let $S_{0}$ be
the set of all finite words $w=a_{1}a_{2}\ldots a_{n}$ with $n\geq1$
and $a_{i}\in\mathbb{A}$. The quantity $n$ is called the length
of $w$ and denoted $\left|w\right|$. The set $S_{0}$ is naturally
a semigroup under the operation of concatenation of words. We will
denote the empty word by $\theta.$ For each $u\in S_{0}$ and $a\in\mathbb{A}$,
we let $\left|u\right|_{a}$ be the number of occurrences of $a$
in $u$. We will identify the elements of $\mathbb{A}$ with the length-one
words over $\mathbb{A}$.

Let $v$ (a variable) be a letter not belonging to $\mathbb{A}$.
By a variable word over $\mathbb{A}$ we mean a word $w$ over $\mathbb{A}\cup\left\{ v\right\} $
with $\left|w\right|_{v}\geq1$. We let $S_{1}$ be the set of variable
words over $\mathbb{A}$. If $w\in S_{1}$ and $a\in\mathbb{A}$,
then $w\left(a\right)\in S_{0}$ is the result of replacing each occurrence
of $v$ by $a$.

A finite coloring of a set $A$ is a function from $A$ to a finite
set $\left\{ 1,2,\ldots,n\right\} $. A subset $B$ of $A$ is monochromatic
if the function is constant on $B$. If $\mathbb{A}$ is any finite
nonempty set and $S$ is the free semigroup of all words over the
alphabet $\mathbb{A}$, then the Hales-Jewett Theorem states that
for any finite coloring of the $S$ there is a variable word over
$\mathbb{A}$ all of whose instances are the same color.
\begin{thm}
\label{HJ}\cite{key-1} Assume that $\mathbb{A}$ is finite. For
each finite coloring of $S_{0}$ there exists a variable word $w$
such that $\left\{ w\left(a\right):a\in\mathbb{A}\right\} $ is monochromatic.
\end{thm}

Now, we need to recall some definitions from \cite{key-4} which are
useful in our work.
\begin{defn}
\cite[Definition 2]{key-4} Let $n\in\mathbb{N}$ and $v_{1},v_{2},\ldots,v_{n}$
be distinct variables which are not members of $\mathbb{A}$.

(a) An $n$-variable word over $\mathbb{A}$ is a word $w$ over $\mathbb{A}\cup\left\{ v_{1},v_{2},\ldots,v_{n}\right\} $
such that $\left|w\right|_{v_{i}}\geq1$ for each $i\in\left\{ 1,2,\ldots,n\right\} $.

(b) If $w$ is an $n$-variable word over $\mathbb{A}$ and $\vec{x}=\left(x_{1},x_{2},\ldots,x_{n}\right)\in\mathbb{A}^{n}$,
then $w\left(\vec{x}\right)$ is the result of replacing each occurrence
of $v_{i}$ in $w$ by $x_{i}$ for each $i\in\left\{ 1,2,\ldots,n\right\} $.

(c) If $w$ is an $n$-variable word over $\mathbb{A}$ and $u=a_{1}a_{2}\ldots a_{n}$
is a length $n$ word, then $w\left(u\right)$ is the result of replacing
each occurrence of $v_{i}$ in $w$ by $a_{i}$ for each $i\in\left\{ 1,2,\ldots,n\right\} $. 

(d) $S_{n}$ is the set of $n$-variable words over $\mathbb{A}$.
\end{defn}

\begin{defn}
Let $S,T$ be two semigroups (or partial semigroups) and let $\nu:T\rightarrow S$
be a homomorphism. Then $\nu$ is called $S$-preserving if 
\[
\nu\left(uw\right)=u\nu\left(w\right)\text{ and }\nu\left(wu\right)=\nu\left(w\right)u
\]
 for every $u\in S$ and every $w\in T$.
\end{defn}

As an example, let $\mathbb{A}$ be any nonempty set and $n\in\mathbb{N}$.
Then for any $\vec{a}\in\mathbb{A}^{n}$, the map $h_{\vec{a}}:S_{n}\rightarrow S_{0}$
defined by $h_{\vec{a}}\left(w\right)=w\left(\vec{a}\right)$ is an
$S_{0}$-preserving homomorphism.
\begin{defn}
Let $S,T\text{ and }R$ be semigroups (or partial semigroups) such
that $S\cup T$ is a semigroup (or partial semigroup) and $T$ is
an ideal of $S\cup T$. Then a homomorphism $\tau:T\rightarrow R$
is said to be $S$-independent if, for every $w\in T$ and every $u\in S$,
\[
\tau\left(uw\right)=\tau\left(w\right)=\tau\left(wu\right).
\]
\end{defn}

As an example, let $T$ be a semigroup with identity $e$. Then for
any $n\geq1$, a homomorphism $\tau:S_{n}\cup S_{0}\rightarrow T$
is $S_{0}$-independent if $\tau\left[S_{0}\right]=\left\{ e\right\} $.

The following is a version of a multivariable extension of the Hales-Jewett
Theorem:
\begin{thm}
\label{MHJ}Assume that $\mathbb{A}$ is finite. Let $S_{0}$ be finitely
colored and let $n\in\mathbb{N}$. There exists $w\in S_{n}$ such
that $\left\{ w\left(\vec{x}\right):\vec{x}\in\mathbb{A}^{n}\right\} $
is monochromatic.
\end{thm}

\begin{rem}
The multivariable version of the Hales-Jewett theorem follows from
the Hales-Jewett theorem itself.
\end{rem}

We need to use some elementary structure of adequate partial semigroups. 
\begin{defn}
A partial semigroup is defined as a pair $\left(G,\ast\right)$ where
$\ast$ is an operation defined on a subset $X$ of $G\times G$ and
satisfies the statement that for all $x,y,z$ in $G$, $\left(x\ast y\right)\ast z=x\ast\left(y\ast z\right)$
in the sense that if either side is defined, so is the other and they
are equal.
\end{defn}

If $\left(G,\ast\right)$ is a partial semigroup, we will denote it
by $G$, when the operation $\ast$ is clear from the context. Now,
we give an example which will be useful in our work.
\begin{example}
Let us consider any sequence $\left\langle x_{n}\right\rangle _{n=1}^{\infty}$
in $\omega$ and let 
\[
G=\text{FS}\left(\left\langle x_{n}\right\rangle _{n=1}^{\infty}\right)=\left\{ \sum_{j\in H}x_{j}:H\in\mathcal{P}_{f}\left(\mathbb{N}\right)\right\} 
\]
 and 
\[
X=\left\{ \left(\sum_{j\in H_{1}}x_{j},\sum_{j\in H_{2}}x_{j}\right):H_{1}\cap H_{2}=\emptyset\right\} .
\]
 Define $\ast:X\rightarrow G$ by 
\[
\left(\sum_{j\in H_{1}}x_{j},\sum_{j\in H_{2}}x_{j}\right)\longrightarrow\sum_{j\in H_{1}}x_{j}+\sum_{j\in H_{2}}x_{j}.
\]
 It is easy to check that $G$ is a commutative partial semigroup.
One can similarly check the same for 
\[
G=\text{FP}\left(\left\langle x_{n}\right\rangle _{n=1}^{\infty}\right)=\left\{ \prod_{j\in H}x_{j}:H\in\mathcal{P}_{f}\left(\mathbb{N}\right)\right\} .
\]
\end{example}

\begin{defn}
\cite[Definition 2.1]{key-2} Let $\left(G,*\right)$ be a partial
semigroup. 

(a) For $g\in G$, $\varphi\left(g\right)=\left\{ h\in G:g\ast h\text{ is defined}\right\} $.

(b) For $H\in\mathcal{P}_{f}\left(G\right)$, $\sigma\left(H\right)=\bigcap_{h\in H}\varphi\left(h\right)$.

(c) For $g\in G$ and $A\subseteq G$, $g^{-1}A=\left\{ h\in\varphi\left(g\right):g*h\in A\right\} $.

(d) $\left(G,\ast\right)$ is adequate if and only if $\sigma\left(H\right)\neq\emptyset$
for all $H\in\mathcal{P}_{f}\left(G\right)$.
\end{defn}

The adequate property is very interesting and very useful to us. As
we will work with the sets of the form $\text{FS}\left(\left\langle x_{n}\right\rangle _{n=1}^{\infty}\right),$
and it is an adequate partial semigroup, we do not need general partial
semigroup. Throughout this paper, we have used elementary combinatorics
to characterize the Hales-Jewett type theorems on small sets, like
$J$-sets, $C$-sets etc., which can only be defined in a partial
semigroup if the partial semigroup is adequate. These type of sets
are called ``small'' as they need not be piecewise syndetic or need
not have positive density. For details the readers can see \cite{key-5}.
For a semigroup $S$, let $^{\mathbb{N}}S$ be the set of all sequences
in $S$ and let,
\[
\mathcal{J}_{m}=\left\{ t=\left(t_{1},t_{2},\ldots,t_{m}\right)\in\mathbb{N}^{m}:t_{1}<t_{2}<...<t_{m}\right\} .
\]

\begin{defn}
\label{J}Let $\left(S,\cdot\right)$ be a semigroup and $A\subseteq S$.
Then $A$ is a $J$-set if and only if for each $F\in\mathcal{P}_{f}\left(^{\mathbb{N}}S\right)$
there exist $m\in\mathbb{N}$, $a=\left(a_{1},a_{2},\ldots,a_{m+1}\right)\in S^{m+1}$
and $t=\left(t_{1},t_{2},\ldots,t_{m}\right)\in\mathcal{J}_{m}$ such
that for each $f\in F$, 
\[
\left(\prod_{j=1}^{m}a_{j}\cdot f\left(t_{j}\right)\right)\cdot a_{m+1}\in A.
\]
\end{defn}

A $C$-set is a set that satisfies the conclusion of the central sets
theorem. $C$-sets can also be characterized as members of idempotents
in $J\left(S\right)$, the closed ideal containing elements of $\beta S$
(the Stone-\v{C}ech compactification of a semigroup $S$), whose
members are $J$-sets, for more details, readers can see \cite{key-3}.
Using this characterization of $C$-sets in terms of idempotents one
can prove the combinatorial characterization, which will be needed
for our purpose, stated below.
\begin{thm}
\cite[Theorem 14.27, p-358]{key-3}\label{C} Let $\left(S,\cdot\right)$
be an countable infinite semigroup and let $A\subseteq S$. Then the
followings are equivalent.
\begin{enumerate}
\item $A$ is a $C$-set.
\item There is a decreasing sequence $\left\langle D_{n}\right\rangle _{n=1}^{\infty}$
of subsets of $A$ such that 

(i) for each $n\in\mathbb{N}$ and each $x\in D_{n}$, there exists
$m\in\mathbb{N}$ with 
\[
D_{m}\subseteq x^{-1}D_{n}
\]
 and 

(ii) for each $n\in\mathbb{N}$, $D_{n}$ is a $J$ -set.
\end{enumerate}
\end{thm}

The following definition is essential to define $J$-sets in adequate
partial semigroups.
\begin{defn}
\cite[Definition 2.4]{key-2} Let $\left(S,*\right)$ be a partial
semigroup and let $f$ be a sequence in $S$. Then $f$ is adequate
if and only if 

(1) for each $H\in\mathcal{P}_{f}\left(\mathbb{N}\right)$, $\prod_{t\in H}f\left(t\right)$
is defined and 

(2) for each $F\in\mathcal{P}_{f}\left(S\right)$, there exists $m\in\mathbb{N}$
such that $\text{FP}\left(\left\langle f\left(t\right)\right\rangle _{t=m}^{\infty}\right)\subseteq\sigma\left(F\right)$.
\end{defn}

Now, we need to recall the definition of $J$-set for adequate partial
semigroups. First, let $\mathcal{F}$ be the set of all adequate sequences
in $S$.
\begin{defn}
\cite[Definition 3.1(b)]{key-2} Let $G$ be an adequate partial semigroup.
Then a set $A\subseteq G$ is a $J$-set if and only if for all $F\in\mathcal{P}_{f}\left(\mathcal{F}\right)$
and all $L\in\mathcal{P}_{f}\left(G\right)$, there exist $m\in\mathbb{N}$,
$a=\left(a_{1},a_{2},\ldots,a_{m+1}\right)\in G^{m+1}$ and $t=\left(t_{1},t_{2},\ldots,t_{m}\right)\in\mathcal{J}_{m}$
such that for all $f\in F$, 
\[
\left(\prod_{i=1}^{m}a_{i}\ast f\left(t_{i}\right)\right)\ast a_{m+1}\in A\cap\sigma\left(L\right).
\]
\end{defn}

Our main result is to provide an elementary proof of a generalization
of \cite[Theorem 17]{key-4} for $J$-sets. Also, we will show that
\[
\left\{ w\in S_{n}:\left(\forall\nu\in F\right)\left(\nu\left(w\right)\in D\right)\right\} 
\]
is a $J$-set, whenever $D$ is a $J$-set in $S_{0}$, where $F$
be a finite nonempty set of $S_{0}$-preserving homomorphisms from
$S_{n}$ into $S_{0}$.

\section{Our results}

The following lemma is very useful in our work.
\begin{lem}
\label{1}Let $T$ be a semigroup and $S$ be a subsemigroup of $T$.
Let $F$ be a finite nonempty set of homomorphisms from $T$ to $S$
which are the identity mapping on $S$, i.e. $\nu(s)=s$ for all $s\in S$
and $\nu\in F$. Let $D$ be a $J$-set of $S$, then $\bigcap_{\nu\in F}\nu^{-1}\left[D\right]$
is a $J$-set in $T$.
\end{lem}

\begin{proof}
Let $E\in\mathcal{P}_{f}$$\left(^{\mathbb{N}}T\right)$ and let $G\in\mathcal{P}_{f}\left(^{\mathbb{N}}S\right)$
be defined as 
\[
G=\left\{ \nu\left(f\right):\,f\in E,\,\nu\in F\right\} .
\]
Now as $D\subseteq S$ is a $J$-set, there exists a natural number
$n$, 
\[
a=\left(a_{1},a_{2},\ldots,a_{n+1}\right)\in S^{n+1}
\]
 and $t\in\mathcal{J}_{n}$ such that, for all $\nu\in F$ and for
all $f\in E$ we have 
\[
a_{1}\nu\left(f\right)\left(t_{1}\right)a_{2}\nu\left(f\right)\left(t_{2}\right)\ldots a_{n}\nu\left(f\right)\left(t_{n}\right)a_{n+1}\in D.
\]
 Hence, for all $\nu\in F$ and for all $f\in E$ we have 
\[
\nu\left(a_{1}\right)\nu\left(f\right)\left(t_{1}\right)\nu\left(a_{2}\right)\nu\left(f\right)\left(t_{2}\right)\ldots\nu\left(a_{n}\right)\nu\left(f\right)\left(t_{n}\right)\nu\left(a_{n+1}\right)\in D,
\]
 as $\nu\left(s\right)=s$ for all $s\in S$. Also, since $\nu$ is
a homomorphism, this implies that, for all $f\in E$ we have 
\[
a_{1}f\left(t_{1}\right)a_{2}f\left(t_{2}\right)\ldots a_{n}f\left(t_{n}\right)a_{n+1}\in\nu^{-1}\left[D\right].
\]
 Thus, $\bigcap_{\nu\in F}\nu^{-1}\left[D\right]$ is a $J$-set.
\end{proof}
\begin{fact}
\label{fact} Let the homomorphism $\tau:S_{0}\cup S_{1}\rightarrow\omega$
be defined by $\tau\left(w\right)=\left|w\right|_{v}$. So, $\tau\left(w\right)=0\text{ if and only if }w\in S_{0}$.
Let $D\subseteq S_{0}$ is a $J$-set and let $\left\langle x_{n}\right\rangle _{n=1}^{\infty}$
be a sequence in $\mathbb{N}$. Let $y_{1}=0,\,y_{n+1}=x_{n}$, for
all $n\in\mathbb{N}$. Then $A=\text{FS}\left(\left\langle y_{n}\right\rangle _{n=1}^{\infty}\right)$
is an adequate partial semigroup. As, $\tau$ is a homomorphism, one
can easily check that $\tau^{-1}\left[A\right]$ is also an adequate
partial semigroup.
\end{fact}

\begin{lem}
\label{newlemma}Let $T$ be an adequate partial semigroup and $S$
be any adequate partial subsemigroup of $T$. Let $F$ be a nonempty
finite set of partial semigroup homomorphisms from $T$ to $S$, which
are $S$-preserving and the identity mapping on $S$, i.e. $\nu(s)=s$
for all $s\in S$ and $\nu\in F$. Then, for for any $J$-set $D\subseteq S$,
$\bigcap_{\nu\in F}\nu^{-1}\left[D\right]$ is a $J$-set in $T$.
\end{lem}

\begin{proof}
Proceeding similar to the proof of Lemma \ref{1}, one can obtain
this. So we omit the proof.
\end{proof}
One can also check that $S_{0}$ is not a $J$-set in $S_{1}\cup S_{0}$.
So for any $J$-set $B$ in $S_{1}\cup S_{0}$, $B\setminus S_{0}$
is a $J$-set in $S_{1}\cup S_{0}$.
\begin{thm}
Let $\tau:T=S_{0}\cup S_{1}\rightarrow\omega$ be defined by $\tau\left(w\right)=\left|w\right|_{v}$.
Let $D\subseteq S_{0}$ is a $J$-set and let $\left\langle x_{n}\right\rangle _{n=1}^{\infty}$
be a sequence in $\mathbb{N}$. Then there exists $w\in S_{1}$ such
that $\left\{ w\left(a\right):a\in\mathbb{A}\right\} \subseteq D$
and 
\[
\tau\left(w\right)\in\text{FS}\left(\left\langle x_{n}\right\rangle _{n=1}^{\infty}\right).
\]
\end{thm}

\begin{proof}
Let $y_{1}=0,\,y_{n+1}=x_{n}$, for all $n\in\mathbb{N}$ and let $A=\text{FS}\left(\left\langle y_{n}\right\rangle _{n=1}^{\infty}\right)$.
Let $\left\{ \bar{h}_{a}:a\in\mathbb{A}\right\} $ be a finite set
of partial semigroup homomorphisms from $\tau^{-1}\left[A\right]$
to $S_{0}$ defined by, 
\[
\bar{h}_{a}\left(w\right)=\left\{ \begin{array}{cc}
w\left(a\right) & \text{if }w\in S_{1}\cap\tau^{-1}\left[A\right]\\
w & \text{if }w\in S_{0}\cap\tau^{-1}\left[A\right]
\end{array}\right..
\]

Now, let $D\subseteq S_{0}$ be a $J$-set then $\bigcap_{a\in\mathbb{A}}\bar{h}_{a}^{-1}\left[D\right]$
is a $J$-set in $\tau^{-1}\left[A\right]$ by lemma \ref{newlemma}.
As, $S_{0}$ is not a $J$-set in $\tau^{-1}\left[A\right]$, there
exists $w\in S_{1}\cap\tau^{-1}\left[A\right]$ such that $w\in\bigcap_{a\in\mathbb{A}}\bar{h}_{a}^{-1}\left[D\right]$
. So, $\tau\left(w\right)\in A=\text{FS}\left(\left\langle y_{n}\right\rangle _{n=1}^{\infty}\right)$.
As. $w\in S_{1},\,\tau\left(w\right)\neq0$. So, $\tau\left(w\right)\in\text{FS}\left(\left\langle x_{n}\right\rangle _{n=1}^{\infty}\right)$
and $\bar{h}_{a}\left(w\right)\in D$ for all $a\in\mathbb{A}$.
\end{proof}
The following is a version of the multidimensional Hales-Jewett Theorem.
\begin{thm}
\label{3}Let $S_{n}$ be the set of all $n$-variable words and $T=S_{n}\cup S_{0}$.
If $D\subseteq S_{0}$ is a $J$-set then there exists an $n$-variable
word $w\in S_{n}$ such that, $w\left(\overrightarrow{a}\right)\in D\,\text{for all }\overrightarrow{a}\in\mathbb{A}^{n}$.
\end{thm}

\begin{proof}
Let $F=\left\{ h_{\vec{a}}:\vec{a}\in\mathbb{A}^{n}\right\} $ be
a finite set of homomorphisms from $T$ to $S_{0}$ defined by, 
\[
h_{\vec{a}}\left(w\right)=\left\{ \begin{array}{cc}
w\left(\vec{a}\right) & \text{if }w\in S_{n}\\
w & \text{if }w\in S_{0}
\end{array}\right..
\]
satisfying condition of Lemma \ref{1}. So, $\bigcap_{h_{\vec{a}}\in F}h_{\vec{a}}^{-1}\left[D\right]$
is a $J$-set.

Now, it is easy to check that $S_{0}$ is not a $J$-set in $T$.
So, $S_{n}\cap\bigcap_{h_{\vec{a}}\in F}h_{\vec{a}}^{-1}\left[D\right]\neq\emptyset$
and hence the proof is done.
\end{proof}
The following is a corollary of the Theorem \ref{3}.
\begin{cor}
Let $T$ be a countably infinite semigroup and $S$ be a subsemigroup
of $T$. Let $F$ be a finite nonempty set of homomorphisms from $T$
to $S$ which are the identity mapping on $S$, i.e. $\nu(s)$=s for
all $s\in S$ and $\nu\in F$. Let $D$ be a $C$-set in $T$, then
there exists an infinite sequence $\left\langle w_{n}\right\rangle _{n=1}^{\infty}$
such that for each $H\in\mathcal{P}_{f}\left(\mathbb{N}\right)$ and
for any arbitrary function $\phi:H\rightarrow F,\,\prod_{t\in H}\phi\left(t\right)\left(w_{t}\right)\in D$,
where the product is computed in increasing order of indices.
\end{cor}

\begin{proof}
As $D$ is a $C$-set, by Theorem \ref{C} there is a decreasing sequence
$\left\langle D_{n}\right\rangle _{n=1}^{\infty}$ of subsets of $D$
such that 
\begin{enumerate}
\item for each $n\in\mathbb{N}$ and each $x\in D_{n}$, there exists $m\in\mathbb{N}$
with $D_{m}\subseteq x^{-1}D_{n}$ and 
\item for each $n\in\mathbb{N}$, $D_{n}$ is a $J$-set.
\end{enumerate}
Let $w_{1}\in\bigcap_{\nu\in F}\nu^{-1}\left[D_{1}\right]$ and assume
for $m\in\mathbb{N}$, we have chosen $\left\langle w_{t}\right\rangle _{t=1}^{m}$
from $\bigcap_{\nu\in F}\nu^{-1}\left[D_{1}\right]$ in such a way
that when $\emptyset\neq H\subseteq\left\{ 1,2,\ldots,m\right\} $
and $\varphi:H\rightarrow F,\,\prod_{t\in H}\phi\left(t\right)\left(w_{t}\right)\in D_{1}$. 

Let, $E=\left\{ \prod_{t\in H}\phi\left(t\right)\left(w_{t}\right):\emptyset\neq H\subseteq\left\{ 1,2,\ldots,m\right\} ,\,\phi:H\rightarrow F\right\} $.

Then, $E\subseteq D_{1}$ and let $R=\bigcap_{y\in E}y^{-1}D_{1}$.

Hence, $R\supseteq D_{m}$ for some $m\in\mathbb{N}$ and so it is
a $J$-set. Let, $w_{m+1}\in\bigcap_{\nu\in F}\nu^{-1}\left[D_{m}\right]$.

Now to verify the induction hypothesis, let $\emptyset\neq H\subseteq\left\{ 1,2,\ldots,m+1\right\} $
and let $\phi:H\rightarrow F$. If $m+1\notin H$, the conclusion
holds by the assumption, and so assume that $m+1\in H$. If $H=\left\{ m+1\right\} $,
then $w_{m+1}\in\phi(m+1)^{-1}[D_{m}]$, so assume that $\left\{ m+1\right\} \subset H$
and let $G=H\setminus\left\{ m+1\right\} $. Let $y=\prod_{t\in G}\phi\left(t\right)\left(w_{t}\right).$
Then $w_{m+1}\in\phi(m+1)^{-1}[y^{-1}D_{m}]$ and so $\prod_{t\in H}\phi\left(t\right)\left(w_{t}\right)=y\phi\left(m+1\right)\left(w_{m+1}\right)\in D_{m}\subseteq D_{1}.$

This completes the proof.
\end{proof}
\begin{thm}
\label{16}Let $k,n\in\mathbb{N}$ with $k<n$ and let $T$ be the
set of words over $\left\{ v_{1},v_{2},\ldots,v_{k}\right\} $ in
which $v_{i}$ occurs for each $i\in\left\{ 1,2,\ldots,k\right\} $.
Given $w\in S_{n}$, let $\tau\left(w\right)$ be obtained from $w$
by deleting all occurrences of elements of $\mathbb{A}$ as well as
all occurrences of $v_{i}$ for $k<i\leq n$. Let $\left\langle y_{t}\right\rangle _{t=1}^{\infty}$
be a sequence in $T$ and let $D\subseteq S_{0}$ be a $J$-set of
$S_{0}$ then there exists $w\in S_{n}$ such that $w\left(\overrightarrow{a}\right)\in D$
for all $\overrightarrow{a}\in\mathbb{A}^{n}$ and $\tau\left(w\right)\in\text{FP}\left(\left\langle y_{t}\right\rangle _{t=1}^{\infty}\right)$.
\end{thm}

\begin{proof}
Let $T^{*}=T\cup\left\{ \theta\right\} $, and let $\tau^{*}:S_{n}\cup S_{0}\rightarrow T^{*}$
defined by
\[
\tau^{*}\left(w\right)=\left\{ \begin{array}{c}
\tau\left(w\right)\text{ if }w\in S_{n}\\
\theta\,\,\,\,\,\text{ if }w\in S_{0}
\end{array}\right..
\]
Clearly $\left(\tau^{*}\right)^{-1}\left[\text{FP}\left(\left\langle y_{t}\right\rangle _{t=1}^{\infty}\right)\cup\theta\right]$
is a partial semigroup in $S_{n}$. We know that $F=\left\{ h_{\vec{a}}:\vec{a}\in\mathbb{A}^{n}\right\} $
be a nonempty finite set of partial semigroup $S_{0}$-preserving
homomorphisms on $S_{0}$. Note that $\left(\tau^{*}\right)^{-1}\left[\theta\right]=S_{0}$.
Let $D$ be a $J$-set on $S_{0}$. Then, $\bigcap_{\vec{a}\in\mathbb{A}^{n}}h_{\vec{a}}^{-1}\left[D\right]$
is a $J$-set in $S_{n}\cup S_{0}$. 

Since, $S_{0}$ is not a $J$-set
in $\left(\tau^{*}\right)^{-1}\left[\text{FP}\left\langle y_{t}\right\rangle _{t=1}^{\infty}\cup\theta\right]$,
we have  $\bigcap_{\vec{a}\in\mathbb{A}^{n}}h_{\vec{a}}^{-1}\left[D\right]\setminus S_{0}$
is a $J$-set in  $\left(\tau^{*}\right)^{-1}\left[\text{FP}\left\langle y_{t}\right\rangle _{t=1}^{\infty}\cup\theta\right]$.
So, there exists $w\in S_{n}\cap\bigcap_{\vec{a}\in\mathbb{A}^{n}}h_{\vec{a}}^{-1}\left[D\right]$
such that $\tau^{*}\left(w\right)=\tau\left(w\right)\in\text{FP}\left(\left\langle y_{t}\right\rangle _{t=1}^{\infty}\right)$
( since, $w\in S_{n}$). This completes the theorem.
\end{proof}
In the next theorem we prove \cite[Theorem 17]{key-4} for $J$-sets.
As in \cite[Theorem 17]{key-4}, in Theorem \ref{17}, the semigroup
$T$ and the matrix $M$ satisfy all the appropriate hypotheses for
matrix multiplication to make sense and be distributive over addition.
\begin{thm}
\label{17}Let $\left(T,+\right)$ be a commutative semigroup with
identity $0$. Let $k,m,n\in\mathbb{N}$, and $M$ be a $k\times m$
matrix. The entries of $M$ come from $\omega$. For $i\in\left\{ 1,2,\ldots,m\right\} $,
let $\tau_{i}$ be an $S_{0}$- independent homomorphism from $S_{n}$
to $T$. Define a function $\psi$ on $S_{n}$ by
\[
\psi\left(w\right)=\left(\begin{array}{c}
\tau_{1}\left(w\right)\\
\tau_{2}\left(w\right)\\
\vdots\\
\tau_{m}\left(w\right)
\end{array}\right),
\]
 with the property that for any collection of IP-sets $\left\{ C_{i}:i\in\left\{ 1,2,\ldots,k\right\} \right\} $
in $T$, there exists $a\in S_{n}$ such that $M\psi\left(a\right)\in\times_{i=1}^{k}C_{i}$.
Let $F$ be a finite nonempty set of $S_{0}$-preserving homomorphisms
from $S_{n}$ to $S_{0}$ and $D\subseteq S_{0}$ is a $J$-set in
$S_{0}$. Let $B_{i}=\text{FS}\left(\left\langle x_{n}^{\left(i\right)}\right\rangle _{n=1}^{\infty}\right)$
for $1\leq i\leq k$ be $k$ IP sets in $T$, then, for each $i\in\left\{ 1,2,\ldots,k\right\} $,
there exists $w\in S_{n}$ such that $\nu\left(w\right)\in D$ for
every $\nu\in F$ and $M\psi\left(w\right)\in\times_{i=1}^{k}B_{i}$.
\end{thm}

\begin{proof}
Let $\phi:S_{n}\cup S_{0}\rightarrow\times_{i=1}^{k}\left(B_{i}\cup\left\{ 0\right\} \right)=B$,
where $\phi\left(w\right)=M\psi\left(w\right)$. Then, as in \cite[Theorem 17]{key-4},
$\phi$ is a homomorphism. Since for $i\in\left\{ 1,2,\ldots,m\right\} $,
$\tau_{i}$ is $S_{0}$-independent, $\phi^{-1}\left[0\right]\supseteq S_{0}$.
Now, clearly, $\phi^{-1}\left(B\right)$ is a partial semigroup. Let,
$D_{j}=\left\{ \vec{a}\in B:a_{i}=0\right\} $, i.e; $D_{j}$ contains
all elements of $B$ which have $i^{\text{th}}$ coordinate $0$.
Then, one can easily check that $\phi^{-1}\left[D_{j}\right]$ is
not a $J$-set in $\phi^{-1}\left[B\right]$, as $S_{0}$ is not a
$J$-set in $S_{n}\cup S_{0}$. So, as in proof of Theorem \ref{16},
there exists $w\in S_{n}$ such that $\nu\left(w\right)\in D$ and
$\phi\left(w\right)\in\times_{i=1}^{k}B_{i}$ for all $\nu\in F$.

This completes the proof.
\end{proof}
As a consequence of Theorem \ref{17}, whenever $D\subseteq S_{0}$
is a $J$-set in $S_{0}$ and $n\in\mathbb{N}$, there exists $w\in S_{n}$
such that $\left\{ w\left(\vec{x}\right):\vec{x}\in\mathbb{A}^{n}\right\} \subseteq D$.
\begin{thm}
Let $n\in\mathbb{N}$ and let $D\subseteq S_{0}$ be a $J$-set in
$S_{0}$. Let $F$ be a finite nonempty set of $S_{0}$-preserving
homomorphisms from $S_{n}$ into $S_{0}$. Then 
\[
\left\{ w\in S_{n}:\left(\forall\nu\in F\right)\left(\nu\left(w\right)\in D\right)\right\} 
\]
 is a $J$-set in $S_{n}$.
\end{thm}

\begin{proof}
Let $T=S_{n}\cup S_{0}$ and extend each $\nu\in F$ to all of $T$
by defining $\nu$ to be the identity mapping on $S_{0}$, i.e; $\nu\left(s\right)=s$
for all $s\in S_{0}$. Since, $D\subseteq S_{0}$ is $J$-set, $\bigcap_{\nu\in F}\nu^{-1}\left[D\right]$
is a $J$-set in $S_{n}\cup S_{0}$. Since, $S_{0}$ is not a $J$-set
in $S_{n}\cup S_{0}$, we have $\bigcap_{\nu\in F}\nu^{-1}\left[D\right]\setminus S_{0}$
is a $J$-set in $S_{n}$. So, $\left\{ w\in S_{n}:w\in\bigcap_{\nu\in F}\nu^{-1}\left[D\right]\right\} $
is a $J$-set in $S_{n}$. Thus, $\left\{ w\in S_{n}:\left(\forall\nu\in F\right)\left(\nu\left(w\right)\in D\right)\right\} $
is a $J$-set in $S_{n}$.
\end{proof}
$\vspace{.1in}$

\textbf{Acknowledgment:} The second author of the paper acknowledges
the grant UGC-NET SRF fellowship with id no. 421333 of CSIR-UGC NET
December 2016. We would like to thank Prof. Dibyendu De for his helpful
comments on this paper. We also acknowledge the helpful comments of
the referees to impove the previous draft of the article.

$\vspace{.1in}$


\begin{thebibliography}{1}
{\normalsize{}\bibitem[1]{key-1} A.W.Hales and R.I.Jewett, Regularity
and positional games, Trans. Amer. Math. Soc. 106 (1963), 222-229.}{\normalsize\par}

{\normalsize{}\bibitem[2]{key-5} N. Hindman},{\normalsize{} Small
sets satisfying the Central Sets Theorem, Integers 9 Supplement (2009), Article 5.}{\normalsize\par}

{\normalsize{}\bibitem[3]{key-2} N. Hindman and K. Pleasant, Central
sets theorem for arbitrary adequate partial semigroups, Topology Proceedings
58 (2021), 183-206.}{\normalsize\par}

{\normalsize{}\bibitem[4]{key-3}N. Hindman and D. Strauss, Algebra
in the Stone-\v{C}ech Compactification: Theory and Applications,
second edition, de Gruyter, Berlin, 2012.}{\normalsize\par}

{\normalsize{}\bibitem[5]{key-4} N. Hindman, D. Strauss and L. Q.
Zamboni, Combining extensions of the Hales-Jewett Theorem with Ramsey
Theory in other structures, The Electronic Journal of Combinatorics
26(4) (2019), \#P4.23.}{\normalsize\par}
\end{thebibliography}
\end{document}